\documentclass[11pt]{amsart}
\usepackage{amsmath}
\newtheorem{thm}{Theorem}[section]
\newtheorem{cor}[thm]{Corollary}
\newtheorem{lem}[thm]{Lemma}
\newtheorem{prop}[thm]{Proposition}
\theoremstyle{definition}

\theoremstyle{remark}

\numberwithin{equation}{section}

\begin{document}
\title[J-Class abelian semigroups]{J-Class abelian semigroups of \\ matrices on
$\mathbb{C}^{n}$ and Hypercyclicity}

\author{Adlene Ayadi and Habib Marzougui}
\
\\
\address{Adlene Ayadi$^{1}$,  University of Gafsa, Faculty of Science of Gafsa, Department of Mathematics, Tunisia; Habib Marzougui$^{2}$, \
University of Carthage, Faculty of Science of Bizerte, Department of Mathematics, Zarzouna. 7021, Tunisia.}
\
\\
\email{habib.marzouki@fsb.rnu.tn; adleneso@yahoo.fr}

\thanks{This work is supported by the research unit: syst\`emes dynamiques et combinatoire: 99UR15-15}

\subjclass[2000]{47A16}

\keywords{Hypercyclic semigroup, locally hypercyclic, extended limit set, dense orbit, semigroup}

\begin{abstract} We give a characterization of hypercyclic finitely generated abelian semigroups of matrices on $\mathbb{C}^{n}$
using the extended limit sets (the J-sets). Moreover we construct for any $n\geq 2$ an abelian semigroup $G$ of GL$(n, \mathbb{C})$ generated
by $n+1$ diagonal matrices which is locally hypercyclic but not hypercyclic and such that J$_{G}(e_{k}) = \mathbb{C}^{n}$ for every $k= 1,\dots, n$,
where $(e_{1},\dots, e_{n})$ is the canonical basis of $\mathbb{C}^{n}$. This gives a negative answer to a question raised by Costakis and
Manoussos.
\end{abstract}
\maketitle

\section{\bf Introduction }
Let $M_{n}(\mathbb{C})$ be the set of all square matrices over $\mathbb{C}$ of order $n\geq 1$ and by GL($n,\mathbb{C})$ the group
of invertible matrices of  $M_{n}(\mathbb{C})$. Let $G$ be a finitely generated
abelian sub-semigroup of $M_{n}(\mathbb{C})$. For a vector $v\in \mathbb{C}^{n}$, we
consider the orbit of $G$ through $v$: $G(v) = \{Av: \ A\in G\}
\subset \mathbb{\mathbb{C}}^{n}$. A subset $E\subset \mathbb{C}^{n}$ is called \emph{$G$-invariant} if
$A(E)\subset E$  for any  $A\in G$. The orbit $G(v)\subset \mathbb{C}^{n}$ is \textit{dense} in ${\mathbb{C}}^{n}$ if $\overline{G(v)}=
{\mathbb{C}}^{n}$, where $\overline{E}$ denotes the closure of a subset $E\subset \mathbb{C}^{n}$. The semigroup $G$ is
called \textit{hypercyclic} if there exists a vector $v\in {\mathbb{C}}^{n}$ such that $G(v)$ is dense
 in ${\mathbb{C}}^{n}$. Hypercyclic is also called topologically transitive.
We refer the reader to the recent book \cite{bm} and \cite{p} for a thorough account on hypercyclicity.
 In \cite{chm2}, Costakis and Manoussos introduced the concept of extended limit set to $G$:
Suppose that $G$ is generated by $p$ matrices $A_{1},\dots,A_{p}$ $(p\geq 1)$ then for $x\in \mathbb{C}^{n}$, we define the extended
limit set J$_{G}(x)$ of $x$ under $G$ to be the set of $y\in \mathbb{C}^{n}$
for which there exists a sequence $(x_{m})_{m}\subset \mathbb{C}^{n}$ and
sequences of non-negative integers $\{k^{(j)}_{m}:\ m\in \mathbb{N}\}$ for $j = 1,2,\dots, p$
with $$ (1.1)\ \ \ \ \ \ k^{(1)}_{m} + k^{(2)}_{m}+\dots+ k^{(p)}_{m}\to+\infty$$
such that ~ $x_{m}\to x$ ~ and ~ $A_{1}^{k^{(1)}_{m}}A_{2}^{k^{(2)}_{m}}\dots A_{p}^{k^{(p)}_{m}}x_{m}\to y$.

Note that condition (1.1) is equivalent to having at least one of the
sequences $\{k^{(j)}_{m}:\ m\in\mathbb{N}\}$  for $j = 1,2,\dots, p$ containing a strictly increasing
subsequence tending to $+\infty$. We say that $G$ is \textit{locally hypercyclic} if there exists a vector
 $v\in {\mathbb{C}}^{n}\backslash\{0\}$ such that J$_{G}(v) = {\mathbb{C}}^{n}$. This notion is a ``localization`` of the
concept of hypercyclicity, this can be justified by the following: J$_{G}(x)= \mathbb{C}^{n}$ if and only if for every open neighborhood
$U_{x}\subset \mathbb{C}^{n}$ of $x$ and every nonempty open set $V\subset \mathbb{C}^{n}$ there exists
 $A\in G$ such that
$A(U_{x})\cap V\neq \emptyset$.

In $\mathbb{C}^{n}$, no matrice can be locally hypercyclic (see \cite{cm2}). However, what is rather remarkable is that in $\mathbb{C}^{n}$ or $\mathbb{R}^{n}$,
a pair of commuting matrices exists which forms a locally hypercyclic, non-hypercyclic semigroup.

The main purpose of this paper is twofold: firstly, we give a characterization of hypercyclic finitely generated abelian semigroup of
$M_{n}(\mathbb{C})$ through the use of the extended limit sets. We show that $G$ is hypercyclic if and only if there exists a vector
$v$ in an open set $V$, defined according to the structure of $G$, such that J$_{G}(v) = {\mathbb{C}}^{n}$ (Theorem \ref{t:1}).
Secondly, we answer negatively (Theorem \ref{t:12}) the following question raised by Costakis and Manoussos in
\cite{chm2}:
is it true that a locally hypercyclic abelian semigroup  $G$ generated by matrices
$A_{1},\dots,A_{p}$ is hypercyclic whenever J$_{G}(u_{k})=\mathbb{C}^{n}$, $k=1,\dots, n$, for a basis $(u_{1},\dots, u_{n})$
of $\mathbb{C}^{n}$? However, we prove that the question is true (see Proposition \ref{p:0}) for any abelian semigroup $G$
consisting of lower triangular matrices on $\mathbb{C}^{n}$ with all diagonal elements equal.
\bigskip
\
\\
Before stating our main results, let introduce the following notations and definitions.
Denote by:
\
\\
$\bullet$ $\mathcal{B}_{0}=(e_{1},\dots,e_{n})$ the canonical basis of $\mathbb{C}^{n}$.
\
\\
$\bullet$ $\mathbb{N}_{0} = \mathbb{N}\backslash\{0\}$.
\
\\
$\bullet$ $I_{n}$ the
identity matrix on $\mathbb{C}^{n}$.
\
\\
Let $n\in\mathbb{N}_{0}$ fixed. For each $m = 1, 2, \dots, n,$ denote by:
\
\\
$\bullet$ \; $\mathbb{T}_{m}(\mathbb{C})$ the set of matrices over $\mathbb{C}$ of the form:

$$\left[\begin{array}{cccc}
          \mu & \ & \ & 0 \\
          a_{2,1} & \ddots & \ & \ \\
          \vdots & \ddots & \ddots & \ \\
          a_{m,1} & \dots & a_{m,m-1} & \mu
        \end{array}
\right]\ \ \ \ (1)$$
\
\\
$\bullet$ \; $\mathbb{T}_{m}^{*}(\mathbb{C}) =
\mathbb{T}_{m}(\mathbb{C})\cap \textrm{GL}(m, \mathbb{C})$ the group of matrices of the form (1) with
$\mu\neq 0$.
\
\\
Let $r\in \mathbb{N}$ and $\eta = (n_{1},\dots,n_{r})$ be a sequence of positive integers such that $n_{1} + \dots + n_{r} = n.$
In particular, $r\leq n$. \
\\
Denote by
\
\\
\textbullet \; \ $\mathcal{K}_{\eta,r}(\mathbb{C}): = \mathbb{T}_{n_{1}}(\mathbb{C})\oplus\dots \oplus \mathbb{T}_{n_{r}}(\mathbb{C}).$
\
\\
\textbullet \; $\mathcal{K}^{*}_{\eta,r}(\mathbb{C}): =
\mathcal{K}_{\eta,r}(\mathbb{C})\cap \textrm{GL}(n, \ \mathbb{C})$.
\
\\
$\bullet$ $U := \underset{k=1}{\overset{r}{\prod}}(\mathbb{C}^{*}\times\mathbb{C}^{n_{k}-1})$.
\
\\
$\bullet$  $v^{T}$ the transpose of a vector $v\in\mathbb{C}^{n}$.
\
\\
$\bullet$ $u_{0} = [e_{1,1},\dots,e_{r,1}]^{T}\in \mathbb{C}^{n}$  where  $e_{k,1} = [1,0,\dots,
0]^{T}\in \mathbb{C}^{n_{k}}$,  $ k=1,\dots, r$.
\
\\
\medskip

Given any abelian sub-semigroup $G$ of $M_{n}(\mathbb{C})$, we introduce the triangular representation for $G$.
\medskip

\begin{prop}\label{p:1}$($\cite{aAhM10}, Proposition 2.2$)$ Let  $G$ be an abelian
sub-semigroup of $M_{n}(\mathbb{C})$. Then there exists a $P\in
\textrm{GL}(n, \mathbb{C})$ such that $PGP^{-1}\subset \mathcal{K}_{\eta,r}(\mathbb{C})$, for some
$\eta\in(\mathbb{N}_{0})^{r}$ and $r\in\{1,\dots,n\}$.
\end{prop}
\
\\
This reduces the existence of a dense orbit to a question concerning sub-semigroups of $\mathcal{K}_{\eta,r}(\mathbb{C})$. For such a choice of matrix $P$,  we let:
\
\\
$\bullet$  $v_{0} = Pu_{0}$.\
\
\\
$\bullet$ $V = P(U)$, it is a dense open set in $\mathbb{C}^{n}$.
\
\\
Our principal results are the following:
\medskip

\begin{thm}\label{t:1} Let $G$ be a finitely generated abelian semigroup of matrices on $\mathbb{C}^{n}$. If
$J_{G}(v)=\mathbb{C}^{n}$ for some $v\in V$ then
$\overline{G(v)} = \mathbb{C}^{n}$.
\end{thm}
\medskip

\begin{cor}\label{c:1} Under the hypothesis of Theorem \ref{t:1}, the following are equivalent:
\begin{itemize}
\item [(i)] $G$ is hypercyclic.
\item [(ii)] $J_{G}(v_{0}) = \mathbb{C}^{n}$.
\item [(iii)] $\overline{G(v_{0})} = \mathbb{C}^{n}$.
\end{itemize}
\end{cor}
\medskip

\begin{cor}\label{c:2} Under the hypothesis of Theorem \ref{t:1},
if $G$ is not hypercyclic then $E: = \{x\in \mathbb{C}^{n}:  J_{G}(x)= \mathbb{C}^{n}\}\subset \underset{k=1}{\overset{r}{\bigcup}}H_{k}$, ~$(r\leq n)$ ~ where ~ $H_{k}$ are $G$-invariant
vector subspaces of $\mathbb{C}^{n}$ with dimension $n-1$.
\end{cor}
\medskip
\
\\
\textbf{Remark}. In the case $n=1$, we have $V=\mathbb{C}^{*}$ and by Theorem ~\ref{t:1}, we conclude that a
sub-semigroup $G$ of $\mathbb{C}$ is hypercylic if and only if it is locally hypercyclic.
\bigskip

\begin{thm}\label{t:12} Let $n\geq 2$ be an integer. Then there exists an abelian semigroup $G$
generated by diagonal matrices $A_{1},\dots,A_{n+1}\in \textrm{GL}(n,\mathbb{C})$ which is not hypercyclic such that
J$_{G}(e_{k})= \mathbb{C}^{n}$ for every $k=1,\dots, n$.
\end{thm}
\medskip

\begin{prop}\label{p:0} Let  $G$ be an abelian
sub-semigroup of $\mathbb{T}_{n}(\mathbb{C})$. If there exists a basis $(e^{\prime}_{1},\dots,e^{\prime}_{n})$ of $\mathbb{C}^{n}$ such that J$_{G}(e^{\prime}_{k}) = \mathbb{C}^{n}$
for every $k=1,\dots, n$, then $G$ is hypercyclic.
\end{prop}
\medskip

\section{\bf Preliminaries and basic notions}
\
\\
Let $G$ be a sub-semigroup of $\mathbb{T}_{n}(\mathbb{C})$. Every element $B\in G$ is written in the form
$$B = \left[\begin{array}{cc}
                       B^{(1)} & 0 \\
                       L_{B} & \mu_{B}
                     \end{array}
\right]$$ with $B^{(1)}\in \mathbb{T}_{n-1}(\mathbb{C}), \ L_{B}\in M_{1,n-1}(\mathbb{C})\ \ \mathrm{and} \ \ \mu_{B}\in\mathbb{C}$.
\
\\
Denote by 
\
\\
- $G^{(1)} = \{B^{(1)}:\ B\in G\}$.
\
\\
- $F_{G} = \textrm{vect}\big(\{(B - \mu_{B}I_{n})e_{i}\in \mathbb{C}^{n}: \ 1\leq i\leq n-1, \
B\in G\}\big)$ the vector subspace generated by the family of vectors
\
\\
$\big\{(B - \mu_{B}I_{n})e_{i}\in \mathbb{C}^{n}: \ 1\leq i\leq n-1, \
B\in G\big\}$.
\
\\
- rank$(F_{G})$ the rank of $F_{G}$. We have rank$(F_{G})\leq n-1$.
\
\\
For every $x = [x_{1},\dots,x_{n}]^{T}\in \mathbb{C}^{n}$, we let
$x^{(1)} = [x_{1},\dots,x_{n-1}]^{T}\in \mathbb{C}^{n-1}$ . We have  $x = [x^{(1)},x_{n}]^{T}$  and $e_{k} = [e^{(1)}_{k},0]^{T}$,
$k = 1,\dots, n-2$.
\
\\
- $F_{G}^{(1)} = \textrm{vect}\big(\{(B^{(1)} - \mu_{B}I_{n-1})e^{(1)}_{k}:\ \ 1\leq k\leq n-2,\ \ B\in G \}\big)$.
\bigskip
\
\\
For every $B\in \mathcal{K}_{n,r}(\mathbb{C})$, write $B = \textrm{diag}
\left(B_{1},\dots, B_{r}\right)$ where $B_{k}\in
\mathbb{T}_{n_{k}}(\mathbb{C})$.
\
\\
If $G$ is an abelian sub-semigroup of $\mathcal{K}_{n,r}(\mathbb{C})$, denote by:
\
\\ ~ $G_{k} = \{B_{k}\ : \ B\in G\}$, it is an abelian
sub-semigroup of $\mathbb{T}_{n_{k}}(\mathbb{C})$.
\
\\
\\
For every $x=[x_{1},\dots,x_{r}]^{T}\in \mathbb{C}^{n}$ where $x_{k}=[x_{k,1},\dots,x_{k,n_{k}}]^{T}\in\mathbb{C}^{n_{k}}$, we let:
\
\\
- $H_{x_{k}} = \mathbb{C}x_{k}+F_{G_{k}}$, $k=1,\dots,r$.
\
\\
- $H_{x} = \underset{k=1}{\overset{r}{\bigoplus}}H_{x_{k}}$.
\
\\
\medskip

We start with the following lemmas:
\medskip

\begin{lem}\label{l:17} Let $G$ be an abelian sub-semigroup of $\mathcal{K}_{\eta,r}(\mathbb{C})$. Under the notations above, for every
$x\in\mathbb{C}^{n}$, $H_{x}$ is $G$-invariant.
\end{lem}
\medskip

\begin{proof} It suffices to prove that $H_{x_{k}}$ is $G_{k}$-invariant: write $x=x_{k}$ and
\
\\
$G=G_{k}$. In this case $H_{x} = \mathbb{C}x+F_{G}$.
 Let $w =[w_{1},\dots,w_{n}]^{T}\in H_{x}$  and $B\in G$ with eigenvalue $\mu$. We have
 $Bw = \mu w+(B-\mu I_{n})w = \mu w+\underset{i=1}{\overset{n-1}{\sum}}w_{k}(B-\mu I_{n})e_{i}$. Since
$w, \ (B-\mu I_{n})e_{i}\in H_{x}$ and $H_{x}$ is a vector space, we have $Bw\in H_{x}$.
\end{proof}
\bigskip

\begin{prop}\label{p:25} Let $G$ be an abelian sub-semigroup of $\mathcal{K}_{\eta,r}(\mathbb{C})$ generated by $A_{1},\dots,A_{p}$.
If $J_{G}(u)=\mathbb{C}^{n}$ for some $u\in U$ then for every $k=1,\dots, r$, $\mathrm{rank}(F_{G_{k}}) = n_{k}-1$.
\end{prop}
\bigskip

\begin{proof} Write $u=[u_{1},\dots,u_{r}]^{T}$, $u_{k}\in \mathbb{C}^{*}\times\mathbb{C}^{n_{k}-1}$,
$A_{j} = \mathrm{diag}(A_{j,1},\dots, A_{j,r})$,  $A_{j,k}\in \mathbb{T}_{n_{k}}$, $j=1,\dots, p$; $k=1,\dots, r$.
First, we will show that 
\
\\
$J_{G_{k}}(u_{k})=\mathbb{C}^{n_{k}}$. For this, let $x_{k}\in \mathbb{C}^{n_{k}}$
 and  $y= [y_{1},\dots, y_{r}]^{T}\in \mathbb{C}^{n}$ such that
 $y_{i}= 0\in \mathbb{C}^{n_{i}}$ if $i\neq k$ and $y_{k}=x_{k}$. As $J_{G}(u) = \mathbb{C}^{n}$, there exist two sequences $(x_{m})_{m}\subset \mathbb{C}^{n}$
  and $(B_{m})_{m}\subset G$ such that $$\underset{m\to +\infty}\lim x_{m}=u\ \ \  \mathrm{and} \ \ \ \underset{m\to +\infty}\lim B_{m}x_{m}
=y.\ \ \ \ \ \ \ (1)$$
\
\\
   Write $x_{m}=[x_{m,1},\dots, x_{m,r}]^{T}$,
  $x_{m,k}\in \mathbb{C}^{n_{k}}$ and $B_{m} = \mathrm{diag}(B_{m,1},\dots,B_{m,r})$, $B_{m,k}\in \mathbb{T}_{n_{k}}(\mathbb{C})$,
  $k=1,\dots, r$. By (1), we have $$\underset{m\to +\infty}\lim x_{m,k}=u_{k}\ \ \ \ \mathrm{and} \ \ \ \underset{m\to +\infty}\lim B_{m,k}x_{m,k}=y_{k}=x_{k}.$$
  Therefore  $x_{k}\in J_{G_{k}}(u_{k})$. It follows that $J_{G_{k}}(u_{k})= \mathbb{C}^{n_{k}}$.
\
\\
Second, one can then suppose that $G\subset \mathbb{T}_{n}(\mathbb{C})$ and $u\in\mathbb{C}^{*}\times\mathbb{C}^{n-1}$. It is clear that $u\notin F_{G}$.
By Lemma \ref{l:17}, $H_{u} = \mathbb{C}u + F_{G}$ is $G$-invariant. Assume that $\mathbb{C}^{n}\backslash H_{u}\neq \emptyset$,
 so let $y\in\mathbb{C}^{n}\backslash H_{u}$.
Then there exist two sequences $(x_{m})_{m}\subset \mathbb{C}^{n}$ and $(B_{m})_{m}\subset G$ such that $\underset{m\to +\infty}\lim x_{m}=u$
and $\underset{m\to +\infty}\lim B_{m}x_{m}=y$. Let
   $H_{x_{m}} = \mathbb{C}x_{m} + F_{G}$ for every $m\in \mathbb{N}$. By Lemma \ref{l:17}, $H_{x_{m}}$ is $G$-invariant, so $B_{m}x_{m}\in H_{x_{m}}$, for
   every $m\in\mathbb{N}$. Write $B_{m}x_{m}=\alpha_{m}x_{m}+z_{m}$, $\alpha_{m}\in\mathbb{C}$ and $z_{m}\in F_{G}$. We distinguish two cases:
\
\\
$\bullet$ If $(\alpha_{m})_{m}$ is bounded, one can suppose by passing to a subsequence, that $(\alpha_{m})_{m\geq 1}$ is convergent,
say $\underset{m\to +\infty}\lim \alpha_{m} =a\in\mathbb{C}$. It follows that
 $\underset{m\to +\infty}\lim z_{m}= y-au\in  F_{G}$ and so $y\in H_{u}$, a contradiction.\
 \\
 \\
   $\bullet$ If $(\alpha_{m})_{m}$ is not bounded, one can suppose by passing to a subsequence, that $\underset{m\to +\infty}\lim|\alpha_{m}| =
+\infty$, then $\underset{m\to +\infty}\lim \frac{1}{\alpha_{m}}z_{m} = -u\in F_{G}$, a contradiction.\
We conclude that $H_{u} = \mathbb{C}^{n}$ and so dim$(F_{G}) = n-1$.
\end{proof}
\bigskip

\section{\bf Proof of Theorem \ref{t:1}, Corollaries \ref{c:1} and \ref{c:2}}
\medskip

Let recall the following results.

\begin{prop}\label{p:2} $($\cite{aAhM06}, Proposition 5.1$)$ Let  $G$ be an abelian sub-semigroup of $\mathbb{T}_{n}(\mathbb{C})$, $n\geq 1$. If
$\mathrm{rank}(F_{G}) = n-1$, then there exists an injective
linear map $\varphi: \mathbb{C}^{n } \ \longrightarrow \
\ \mathbb{T}_{n}(\mathbb{C})$ such that:
\begin{itemize}
 \item [(i)] for every $v\in\mathbb{C}^{n}$,  $\varphi(v)e_{1} = v$
\item [(ii)] $G\subset\varphi(\mathbb{C}^{n})$.
\end{itemize}
  \end{prop}
\
\\
 The map $\varphi$ in Proposition \ref{p:2} can be precise, from the proof of (\cite{aAhM06}, Proposition 5.1), as follows:
\medskip

\begin{cor}\label{C:1} Under the hypothesis of Proposition ~\ref{p:2}, and for $n\geq 2$, there exists a linear map
 $\eta:\mathbb{C}^{n-1}\longrightarrow\mathbb{C}^{n-2}$ such that for every
\
\\
$v= [v_{1},\dots,v_{n}]^{T}\in\mathbb{C}^{n}$,
  $\varphi(v)=\left[\begin{array}{cc}
                       \varphi^{(1)}(v^{(1)}) & 0 \\
                       \left[v_{n}, (\eta(v^{(1)}))^{T}\right]^{T} & v_{1}
                     \end{array}
\right]$, with
\
\\
$v^{(1)} =
[v_{1},\dots,v_{n-1}]^{T}\in\mathbb{C}^{n}$ and $\varphi^{(1)}: \mathbb{C}^{n-1}\longrightarrow
 \mathbb{T}_{n-1}(\mathbb{C})$ is the injective linear map associated to $G^{(1)}$ given by
Proposition ~\ref{p:2}.
\end{cor}
\medskip

\begin{lem}\label{L:12} Let  $G$ be an abelian sub-semigroup of $\mathbb{T}_{n}(\mathbb{C})$, $n\geq 1$. Suppose that rank$(F_{G}) = n-1$.
Let $u, \ v\in \mathbb{C}^{*}\times \mathbb{C}^{n-1}$, $(u_{m})_{m\in\mathbb{N}}$ in
$\mathbb{C}^{*}\times \mathbb{C}^{n-1}$ and $(B_{m})_{m\in\mathbb{N}}$ in $G$ such that $\underset{m\to+\infty}\lim u_{m} = u$ and
$\underset{m\to+\infty}\lim B_{m}u_{m} = v$. If $(B^{(1)}_{m})_{m\in \mathbb{N}}$ is bounded then
$(B_{m})_{m\in \mathbb{N}}$ is bounded.
\end{lem}
\medskip

\begin{proof} By Proposition \ref{p:2}, there exists an injective linear map
\
\\
$\varphi: \mathbb{C}^{n } \ \longrightarrow \
\ \mathbb{T}_{n}(\mathbb{C})$ such that $G\subset\varphi(\mathbb{C}^{n})$ and $\varphi(w)e_{1} = w$ for every $w\in \mathbb{C}^{n}$.
\
\\

Write $B_{m}=\varphi(w_{m})$, $w_{m}=[z_{m,1},\dots,z_{m,n}]^{T}\in\mathbb{C}^{n}$.
By Corollary ~\ref{C:1}, we have $$B_{m}=\varphi(w_{m})=\left[\begin{array}{cc}
\varphi^{(1)}(w^{(1)}_{m}) & 0 \\
\left[z_{m,n}, \ (\eta(w^{(1)}_{m}))^{T}\right]^{T} & z_{m,1}
\end{array}\right],$$
where $\varphi^{(1)}$ is the injective linear map associated to $G^{(1)}$
and
\
\\
$\eta: \mathbb{C}^{n-1}\longrightarrow \mathbb{C}^{n-2}$ is a linear map. Then $w^{(1)}_{m} = \varphi^{(1)}(w^{(1)}_{m})u_{0}^{(1)} =
B^{(1)}_{m}u_{0}^{(1)}$ and thus $(w^{(1)}_{m})_{m\geq 1}$ is bounded. So $(\eta(w^{(1)}_{m}))_{m\geq 1}$ and $(z_{m,1})_{m\geq 1}$ are also
bounded. Write $u_{m}=[u_{m,1},\dots,u_{m,n}]^{T}\in \mathbb{C}^{*}\times \mathbb{C}^{n-1}$. Then
we have $\underset{m\to+\infty}\lim u_{m,1}=u_{1}\neq 0$ and $\underset{m\to+\infty}\lim z_{m,n}u_{m,1} = v_{1}\neq 0$.
Thus  $\underset{m\to+\infty}\lim z_{m,n} = \dfrac{v_{1}}{u_{1}}$ and so $(z_{m,n})_{m\geq 1}$ is bounded. We deduce that $(B_{m})_{m\in \mathbb{N}}$ is bounded, this completes the proof.
\end{proof}
\bigskip

\begin{lem}\label{L:11} Let $G$ be an abelian sub-semigroup of  $\mathbb{T}_{n}(\mathbb{C})$ such that

$\textrm{rank}(F_{G})= n-1$. Let $u, \ v\in \mathbb{C}^{*}\times \mathbb{C}^{n-1}$. If two sequences $(u_{m})_{m\in\mathbb{N}}$ in
$\mathbb{C}^{*}\times \mathbb{C}^{n-1}$ and $(B_{m})_{m\in\mathbb{N}}$ in $G$ such that $\underset{m\to+\infty}\lim u_{m} = u$ and
$\underset{m\to+\infty}\lim B_{m}u_{m} = v$ then $(B_{m})_{m\in\mathbb{N}}$ is bounded.
\end{lem}
\bigskip

\begin{proof} The proof is done by induction on $n$. For $n=1$, we have
\
\\
$u_{m}\in \mathbb{C}^{*}$, $B_{m} = \lambda_{m}\in\mathbb{C}$ and $u, \ v\in\mathbb{C}^{*}$.
 The conditions $\underset{m\to+\infty}\lim u_{m} = u$ and $\underset{m\to+\infty}\lim B_{m}u_{m} = v$
show that  $\underset{m\to+\infty}\lim B_{m} = \dfrac{v}{u}$, hence $(B_{m})_{m\in\mathbb{N}}$ is bounded.
Suppose the Lemma is true up to dimension $n-1$ and let $G$ be an abelian sub-semigroup of $\mathbb{T}_{n}(\mathbb{C})$.
Let $u, \ v\in \mathbb{C}^{*}\times \mathbb{C}^{n-1}$, $(u_{m})_{m\in\mathbb{N}}$ a sequence in $\mathbb{C}^{*}\times \mathbb{C}^{n-1}$ and
$(B_{m})_{m\in\mathbb{N}}$ a sequence in $G$ such that $\underset{m\to+\infty}\lim u_{m}=u$ and

$\underset{m\to+\infty}\lim B_{m}u_{m}=v$.
 We let
 $u = [u_{1},\dots,u_{n}]^{T}, \ v= [v_{1},\dots,v_{n}]^{T}$ and

$u_{m} = [u_{m,1},\dots,u_{m,n}]^{T}\in\mathbb{C}^{n}$. We have
$\underset{m\to+\infty}\lim u^{(1)}_{m} = u^{(1)}$ and
\
\\
$\underset{m\to+\infty}\lim B^{(1)}_{m}u^{(1)}_{m} = v^{(1)}$. The set $G^{(1)}$ is an abelian sub-semigroup of $\mathbb{T}_{n-1}(\mathbb{C})$ and $\textrm{rank}(F_{G^{(1)}})= n-2$. By induction hypothesis
applied to $G^{(1)}$ on $\mathbb{C}^{n-1}$, the sequence $(B^{(1)}_{m})_{m\in\mathbb{N}}$ is bounded.
Therefore, by Lemma \ref{L:12}, $(B_{m})_{m}$ is bounded.
\end{proof}
\bigskip

\begin{cor}\label{c:36} Let $G$ be an abelian sub-semigroup of $\mathcal{K}_{\eta,r}(\mathbb{C})$ generated by $A_{1},\dots,A_{p}$, $p\geq 1$.
Suppose that
$\mathrm{rank}(F_{G_{k}}) = n_{k}-1$, $k=1,\dots, r$. If $x, \ y\in U$ and two sequences
$(B_{m})_{m}\subset G$ and $(x_{m})_{m}\subset \mathbb{C}^{m}$ such that $\underset{m\to+\infty}\lim x_{m}=x$ and $\underset{m\to+\infty}\lim B_{m}x_{m}=y$ then $(B_{m})_{m}$ is bounded.
\end{cor}
\medskip

\begin{prop}\label{p:101}\cite{chm2} Let $G$ be an abelian sub-semigroup of $M_{n}(\mathbb{C})$ generated by $A_{1},\dots,A_{p}$, $p\geq 1$.
Then $G$ is hypercyclic if and only if
 $J_{G}(x)= \mathbb{C}^{n}$ for every $x\in\mathbb{C}^{n}$.
\end{prop}
\medskip

\begin{proof}[Proof of Theorem ~\ref{t:1}] One can assume, by Proposition \ref{p:1}, that $G$
 is a sub-semigroup of $\mathcal{K}_{\eta,r}(\mathbb{C})$. Suppose that $J_{G}(u)= \mathbb{C}^{n}$ where $u\in U$. Then by Proposition \ref{p:25}, $\mathrm{rank}(F_{G_{k}}) = n_{k}-1$, for every $k=1,\dots,r$.
  Let $y\in U$, then there exist two sequences $(B_{m})_{m}\subset G$ and $(x_{m})_{m}\subset
\mathbb{C}^{n}$ satisfying: $$\underset{m\to+\infty}\lim x_{m}=u\
\ \  \mathrm{and}\ \ \ \underset{m\to+\infty}\lim B_{m}x_{m}=y.$$

 So by Corollary \ref{c:36}, $(B_{m})_{m\geq 1}$ is bounded:
$\|B_{m}\|\leq M$ for some $M>0$ where is $\| . \|$ is the Euclidean norm on $\mathbb{C}^{n}$. Then
\begin{align*}
\|B_{m}u-y\|& =\|B_{m}v-B_{m}x_{m}+B_{m}x_{m}-y\|\\
 \ & \leq\|B_{m}u-B_{m}x_{m}\| + \|B_{m}x_{m}-y\|\\
 \ & \leq \|B_{m}\| \|u-x_{m}\|
+ \|B_{m}x_{m}-y\|\\
 \ &  \leq M \|u-x_{m}\|
+ \|B_{m}x_{m}-y\|
\end{align*}
Thus $\underset{m\to+\infty}\lim B_{m}u=y$ and so $y\in
\overline{G(u)}$. It follows that $U\subset\overline{G(u)}$ and as $\overline{U} = \mathbb{C}^{n}$, we get $\overline{G(u)} = \mathbb{C}^{n}$.
\end{proof}
\
\\
\\
{\it Proof of Corollary ~\ref{c:1}.} $(i)\Longrightarrow(ii)$ follows from Proposition ~\ref{p:101}.\
\\
$(ii)\Longrightarrow(iii):$ this results from Theorem \ref{t:1}.
$(iii)\Longrightarrow(i)$: is clear.
\
\\
\\
{\it Proof of Corollary ~\ref{c:2}.} If $G$ is not hypercyclic then by Theorem \ref{t:1},
\
\\
$J_{G}(v)\neq \mathbb{C}^{n}$  for any $v\in V$, thus
$V\cap E = \emptyset$ and therefore $E\subset\mathbb{C}^{n}\backslash V = \underset{k=1}{\overset{r}{\bigcup}}H_{k}$ with
$H_{k} = P\big(\left\{x = [x_{1},\dots, x_{r}]^{T}, \ x_{i}\in\mathbb{C}^{n_{i}},
\ \mathrm{if}\  i\neq k, \mathrm{and }~ x_{k}\in\{0\}\times\mathbb{C}^{n_{k}-1}\right\} \big)$. \qed
\medskip

\section{\bf Proof of Theorem \ref{t:12} and Proposition \ref{p:0}}
\medskip

For the proof of Theorem \ref{t:12}, we will make use of the following result:
\medskip


\begin{lem}\label{l:78}\cite{fe1}. If $a\in \mathbb{C}$ with $|a|> 1$, then there is a dense set
\
\\
$\Delta_{a}\subset \{z\in \mathbb{C},\ \ |z|<1\}$ such that for any
$b\in \Delta_{a}$, we have that
\
\\
$\{a^{k}b^{l}: k,l\in \mathbb{N}\}$ is dense in $\mathbb{C}$.
\end{lem}
\medskip
\
\\
{\it Proof of Theorem ~\ref{t:12}.} Let $a\in \mathbb{C}$ with $|a|>1$. By Lemma \ref{l:78}, there exists
$b\in \mathbb{C}$ with $\dfrac{1}{|a|}<|b|<1$ such that $\{a^{k}b^{l}:\ \ k,l\in\mathbb{N}\}$ is dense in $\mathbb{C}$.
 Consider the abelian sub-semigroup $G$ of $\mathcal{K}^{*}_{\eta,r}(\mathbb{C})$ generated by
$B, A_{1},\dots, A_{n}$, where $$B= bI_{n} ~~~~ \textrm{and} ~~~~ A_{k} = \textrm{diag}(\underset{(k-1)-terms}{\underbrace{a,\dots\dots,a}},1,a\dots,a),\
k=1,\dots, n.$$
\
\\
$\bullet$ First, we will show that $G$ is not hypercyclic: for this, it is equivalent to prove, by Corollary \ref{c:1}, that
$\overline{G(u_{0})}\neq \mathbb{C}^{n}$ where $u_{0}= [1,\dots,1]^{T}$: We have
$$G(u_{0}) = \left\{[b^{m}a^{k_{2}}\dots a^{k_{n}};\ b^{m}a^{k_{1}}a^{k_{3}}\dots a^{k_{n}};\dots\dots;\ b^{m}a^{k_{1}}\dots a^{k_{n-1}}]^{T}:~
m,k_{1},\dots, k_{n}\in\mathbb{N}\right\}$$
Observe that for every $x = [x_{1},\dots,x_{n}]^{T}\in G(u_{0})$, we have $\dfrac{x_{i}}{x_{n}}= a^{k_{n}-k_{i}}$, $i = 1,\dots,n$. So
$G(u_{0})\subset F$ where $$F = \left\{\left[a^{\ell_{1}}\lambda;\ a^{\ell_{2}}\lambda;\dots\dots;\ a^{\ell_{n}}\lambda\right]^{T}: \lambda\in \mathbb{C},
\ \ell_{1},\dots, \ell_{n}\in\mathbb{Z}\right\}.$$
We set $$D_{k} = \textrm{diag}(\underset{(k-1)-\textrm{terms}}{\underbrace{1,\dots\dots,1}},a,1\dots,1),\  k=1,\dots, n$$
and
$$\Delta = \{[\lambda,\dots,\lambda]^{T}:\ \ \ \lambda\in \mathbb{C}\}.$$ Then we have
$$F = \underset{\underset{1\leq i\leq n}{\ell_{i}\in\mathbb{Z}} }{\bigcup}D_{1}^{\ell_{1}}\dots D_{n}^{\ell_{n}}(\Delta).$$

It is plain  that $\overline{F} = F\cup \underset{k=1}{\overset{n}{\bigcup}}E_{k},$ where
$$E_{k}= \{x=[x_{1},\dots,x_{n}]^{T}: \ x_{i}\in\mathbb{C},\ \mathrm{if}\ i\neq k ~\textrm{and} \ x_{k} = 0\}.$$  Since
 $\overset{\circ}{E_{k}}= \emptyset$, it follows that $\overset{\circ}{\overbrace{{\underset{k=1}{\overset{n}{\bigcup}}E_{k}}}} = \emptyset$, where $\overset{\circ}{M}$ denotes the interior of
 a subset $M\subset \mathbb{C}^{n}$. Moreover,
since $\overset{\circ}{\overbrace{D_{1}^{k_{1}}\dots D_{n}^{k_{n}}(\Delta)}} = \emptyset$, it follows, by Baire's theorem, that
$\overset{\circ}{\overline{G(u_{0})}} = \emptyset$ and so $\overline{G(u_{0})}\neq \mathbb{C}^{n}$.
\
\\
\\
$\bullet$ Second, we will show that $J_{G}(e_{k}) = \mathbb{C}^{n}$ for every $k = 1,\dots, n$.
 Fix a vector $y = [y_{1},\dots, y_{n}]^{T}\in \mathbb{C}^{n}$.
Choose two sequences of positive integers $(i_{m})_{m\in\mathbb{N}}$ and
 $(j_{m})_{m\in\mathbb{N}}$ with $i_{m},~ j_{m}\rightarrow +\infty$ such that
$\underset{m\to +\infty}\lim a^{i_{m}}b^{j_{m}}= y_{k}$. As $n\geq 2$, one can choose $s\in\{1,\dots,n\}$ such
that $s\neq k$. We let $B_{m}= A_{s}^{i_{m}}B^{j_{m}}A_{k}^{j_{m}}$. Then we have $B_{m} = \mathrm{diag}(a_{m,1},\dots,a_{m,n})$
\
\\
where
 $$a_{m,l}= \begin{cases} a^{i_{m}}a^{j_{m}}b^{j_{m}} \ & \textrm{if}\ \ l\neq k, s \\
          a^{j_{m}}b^{j_{m}} \ & \textrm{if}\  \ l= s\\
          a^{i_{m}}b^{j_{m}} \ & \textrm{if} \ \ l= k
                 \end{cases}$$
Hence, $\underset{m\to +\infty}\lim a_{m,k}= y_{k} \ (3)$. Moreover, since $\dfrac{1}{|a|}<|b|<1$, we have $\underset{m\to+\infty}\lim a_{m,l} = +\infty$ for every $l\neq k$ and hence
$\underset{m\to+\infty}\lim x_{m} = e_{k}$. We set $x_{m} = \left(x_{m,1},\dots, x_{m,n} \right)$  where
 $$x_{m,l}= \begin{cases}
\frac{y_{l}}{a_{m,l}} \ & \textrm{if}\ \ l\neq k \\
1 \ & \textrm{if} \ \ l= k
\end{cases}$$
Then
$B_{m}x_{m} = \left(y_{1},\dots,y_{k-1},a_{m,k},
y_{k+1},\dots, y_{n}\right)$, and by $(3)$, it follows that
\
\\
$\underset{m\to+\infty}\lim B_{m}x_{m}=y$. We conclude that
$y\in J_{G}(e_{k})$ and therefore $J_{G}(e_{k})= \mathbb{C}^{n}$, for every $k = 1,\dots, n$. \qed
\
\\
\medskip
\
\\
{\it Proof of Proposition ~\ref{p:0}.}  Since $(e^{\prime}_{1},\dots,e^{\prime}_{n})$ is a basis of $\mathbb{C}^{n}$, there exists
$i_{0}\in \{1,\dots, n\}$ such that $e^{\prime}_{i_{0}}\in \mathbb{C}^{*}\times \mathbb{C}^{n-1}$. As
$V = U = \mathbb{C}^{*}\times \mathbb{C}^{n-1}$ and
\
\\
J$_{G}(e^{\prime}_{i_{0}}) = \mathbb{C}^{n}$ then by
Theorem ~\ref{t:1}, $\overline{G(e^{\prime}_{i_{0}})} = \mathbb{C}^{n}$ and hence $G$ is hypercyclic. \qed
\
\\

The following questions arose naturally.
\medskip
\
\\
\textbf{Question 1}. Find analogous to Theorems \ref{t:1} and \ref{t:12} for the real case?
\medskip
\
\\
\textbf{Question 2}. Let $1\leq r\leq n$ be an integer. Is it true that there exists a finitely generated
abelian semigroup $G$ of $\mathcal{K}_{\eta,r}(\mathbb{C})$ which is not hypercyclic such that
J$_{G}(e_{k}) = \mathbb{C}^{n}$ for every $k=1,\dots, n$? Similarly for $\mathbb{R}^{n}$?

Notice that for $r=n$, this question is answered positively (Theorem \ref{t:12}) However, for $r=1$, it is answered negatively (Proposition \ref{p:0}).
\medskip
\
\\

\bibliographystyle{amsplain}
\vskip 0,6 cm

\medskip

\end{document}